\documentclass[a4paper,11pt,oneside]{article}
\usepackage[bookmarks=false]{hyperref}
\usepackage{amsmath,titlesec,amsthm,amssymb}

\titleformat{\section}{\large\bf\boldmath}{\arabic{section}.}{2ex}{}

\titlespacing*{\section}{0ex}{2ex}{1ex}

\titleformat{\subsection}{\bf\boldmath}{\arabic{section}.\arabic{subsection}.}{2ex}{}

\titlespacing*{\subsection}{0ex}{0.8ex}{0ex}

\def\dblone{\hbox{$1\hskip -1.2pt\vrule depth 0pt height 1.6ex width 0.7pt\vrule depth -0.1pt height 0.3pt width 0.12em$}}
\makeatletter

\newcommand{\Rmnum}[1]{\expandafter\@slowromancap\romannumeral #1@}
\makeatother

\newtheorem{thmstar}{Theorem}
\newtheorem{theorem}{Theorem}[section]
\newtheorem{lemma}[theorem]{Lemma}
\newtheorem{proposition}[theorem]{Proposition}

{\theoremstyle{definition}

\newtheorem{remark}[theorem]{Remark}
}

\newcommand{\B}{\operatorname{B}}
\newcommand{\recht}{\rightarrow}
\newcommand{\cZ}{\mathcal{Z}}
\newcommand{\om}{\omega}
\newcommand{\cR}{\mathcal{R}}
\newcommand{\vphi}{\varphi}
\newcommand{\cF}{\mathcal{F}}

\newcommand{\cH}{\mathcal{H}}
\newcommand{\cG}{\mathcal{G}}

\newcommand{\N}{\mathbb{N}}
\newcommand{\cU}{\mathcal{U}}
\newcommand{\Ker}{\operatorname{Ker}}
\newcommand{\Ima}{\operatorname{Im}}
\newcommand{\cJ}{\mathcal{J}}

\newcommand{\Z}{\mathbb{Z}}
\newcommand{\actson}{\curvearrowright}
\newcommand{\Aut}{\operatorname{Aut}}

\newcommand{\C}{\mathbb{C}}
\newcommand{\supp}{\operatorname{supp}}
\newcommand{\graph}{\operatorname{graph}}
\newcommand{\dom}{\operatorname{dom}}
\newcommand{\ran}{\operatorname{ran}}
\newcommand{\al}{\alpha}
\newcommand{\R}{\mathbb{R}}

\newcommand{\Tr}{\operatorname{Tr}}
\newcommand{\M}{\operatorname{M}}
\newcommand{\lspan}{\operatorname{span}}
\newcommand{\ot}{\otimes}
\newcommand{\bim}[3]{\mathord{\raisebox{-0.4ex}[0ex][0ex]{\scriptsize $#1$}{#2}\hspace{-0.05ex}\raisebox{-0.4ex}[0ex][0ex]{\scriptsize $#3$}}}
\newcommand{\eps}{\varepsilon}
\newcommand{\HNN}{\operatorname{HNN}}
\newcommand{\T}{\mathbb{T}}
\newcommand{\F}{\mathbb{F}}
\newcommand{\si}{\sigma}

\newcommand{\dpr}{^{\prime\prime}}

\newcommand{\PIso}{\operatorname{PIso}}
\newcommand{\PAut}{\operatorname{PAut}}
\newcommand{\ram}{\mathcal{R}(A\subset M)}
\newcommand{\rnm}{\mathcal{R}_{n,m}}
\newcommand{\qam}{\operatorname{Q}_M(A)}
\newcommand{\qab}{\operatorname{Q}_M(A,B)}
\newcommand{\BS}{\operatorname{BS}}
\newcommand{\cN}{\mathcal{N}}
\newcommand{\op}{^{\text{\rm op}}}
\newcommand{\QN}{\operatorname{QN}}
\newcommand{\Qu}{\operatorname{Q}}
\newcommand{\cRtil}{\widetilde{\mathcal{R}}}
\newcommand{\Sp}{\operatorname{Sp}}
\newcommand{\ovt}{\mathbin{\overline{\otimes}}}
\newcommand{\cM}{\mathcal{M}}

\begin{document}\begin{center}
{\LARGE\bf\boldmath Partial classification of the Baumslag-Solitar\vspace{0.5ex}\\ group von Neumann algebras}

\vspace{1ex}

by Niels Meesschaert\footnote{KU~Leuven, Department of Mathematics, Leuven (Belgium), niels.meesschaert@wis.kuleuven.be \\
    Supported by ERC Starting Grant VNALG-200749 and Research Programme G.0639.11 of the Research Foundation --
    Flanders (FWO).} and Stefaan Vaes\footnote{KU~Leuven, Department of Mathematics, Leuven (Belgium), stefaan.vaes@wis.kuleuven.be \\
    Supported by ERC Starting Grant VNALG-200749, Research
    Programme G.0639.11 of the Research Foundation --
    Flanders (FWO) and KU Leuven BOF research grant OT/13/079.}
\end{center}

\begin{abstract}\noindent
We prove that the rational number $|n/m|$ is an invariant of the group von Neumann algebra of the Baumslag-Solitar group $\BS(n,m)$. More precisely, if $L(\BS(n,m))$ is isomorphic with $L(\BS(n',m'))$, then $|n'/m'| = |n/m|^{\pm 1}$. We obtain this result by associating to abelian, but not maximal abelian, subalgebras of a II$_1$ factor, an equivalence relation that can be of type III. In particular, we associate to $L(\BS(n,m))$ a canonical equivalence relation of type III$_{|n/m|}$.
\end{abstract}

\section{Introduction and statement of the main result}

Some of the deepest open problems in functional analysis center around the classification of group von Neumann algebras $L(G)$ associated with certain natural families of countable groups $G$. In the case of the free groups, this becomes the famous free group factor problem asking whether $L(\F_n) \cong L(\F_m)$ when $n,m \geq 2$ and $n \neq m$. For property (T) groups with infinite conjugacy classes (icc), this leads to Connes's rigidity conjecture (\cite{Co80}) asserting that an isomorphism $L(G) \cong L(\Lambda)$ between the property (T) factors entails an isomorphism $G \cong \Lambda$ of the groups.

As a consequence of Connes's uniqueness theorem of injective II$_1$ factors (\cite{Co75}), the group von Neumann algebra $L(G)$ of an amenable icc group $G$ is isomorphic with the unique hyperfinite II$_1$ factor $R$. In the nonamenable case, many nonisomorphic groups $G$ are known to have nonisomorphic group von Neumann algebras $L(G)$. Nevertheless, concerning the classification of group von Neumann algebras of natural families of groups, e.g.\ lattices in simple Lie groups, little is known. A notable exception however is \cite{CH88} where it is shown that for $n \neq m$, lattices in $\Sp(n,1)$, respectively $\Sp(m,1)$, have nonisomorphic group von Neumann algebras.

Since 2001, Popa has been developing a new arsenal of techniques to study II$_1$ factors, called deformation/rigidity theory. This theory has provided several classes $\cG$ of groups such that an isomorphism $L(G) \cong L(\Lambda)$ with both $G,\Lambda \in \cG$ entails the isomorphism $G \cong \Lambda$. By \cite{Po04}, this holds in particular when $\cG$ is the class of wreath product groups of the form $(\Z / 2\Z) \wr \Gamma$ with $\Gamma$ an icc property (T) group.

In \cite{IPV10}, the first W$^*$-superrigidity theorems for group von Neumann algebras were discovered, yielding icc groups $G$ such that an isomorphism $L(G) \cong L(\Lambda)$ with $\Lambda$ an \emph{arbitrary} countable group, implies that $G \cong \Lambda$. The groups $G$ discovered in \cite{IPV10} are generalized wreath products of a special form. In \cite{BV12}, it was then shown that one can actually take $G = (\Z/2\Z)^{(\Gamma)} \rtimes (\Gamma \times \Gamma)$ with $\Gamma$ ranging over a large family of nonamenable groups including the free groups $\F_n$, $n \geq 2$.

In this article, we apply Popa's deformation/rigidity theory to partially classify the group von Neumann algebras of the Baumslag-Solitar groups $\BS(n,m)$. Recall that for all $n,m \in \Z - \{0\}$, this group is defined as the group generated by $a$ and $b$ subject to the relation $b a^n b^{-1} = a^m$. So,
$$\BS(n,m) := \langle a,b \mid b a^n b^{-1} = a^m \rangle \; .$$
The Baumslag-Solitar groups were introduced in \cite{BS62} as the first examples of finitely presented non-Hopfian groups. Ever since, they have been used as examples and counterexamples for numerous group theoretic phenomena. Therefore, it is a natural problem to classify the group von Neumann algebras $L(\BS(n,m))$.

Whenever $|n| = 1$ or $|m|=1$, the group $\BS(n,m)$ is solvable, hence amenable. So we always assume that $|n| \geq 2$ and $|m| \geq 2$. In that case, $\BS(n,m)$ contains a copy of the free group $\F_2$ and hence, is nonamenable. In \cite{Mo91}, the Baumslag-Solitar groups were classified up to isomorphism: $\BS(n,m)\cong \BS(n',m')$ if and only if $\{n,m\}=\{\eps n',\eps m'\}$ for some $\eps\in\{-1,1\}$. So, up to isomorphism, we only consider $2 \leq n \leq |m|$. Finally by \cite[Exemple 2.4]{St05}, the group $\BS(n,m)$ is icc if and only if $|n|\neq |m|$. Therefore, we always assume that $2 \leq n < |m|$.

Using Popa's deformation/rigidity theory and in particular his spectral gap rigidity (\cite{Po06}) and the work on amalgamated free products (\cite{IPP05}), several structural properties of the II$_1$ factors $M=L(\BS(n,m))$ were proven. In particular, it was shown in \cite{Fi10} that $M$ is not solid, that $M$ is prime and that $M$ has no Cartan subalgebra. More generally, it is proven in \cite{Fi10} that any amenable regular von Neumann subalgebra of $M$ must have a nonamenable relative commutant.

Our main result is the following partial classification theorem for the Baumslag-Solitar group von Neumann algebras $L(\BS(n,m))$. Whenever $M$ is a II$_1$ factor and $t > 0$, we denote by $M^t$ the \emph{amplification} of $M$. Up to unitary conjugacy, $M^t$ is defined as $p(\M_n(\C) \ot M)p$ where $p$ is a projection satisfying $(\Tr \ot \tau)(p) = t$. The II$_1$ factors $M$ and $N$ are called \emph{stably isomorphic} if there exists a $t > 0$ such that $M \cong N^t$.

\begin{thmstar}\label{thm.A}
Let $n,m,n',m'\in \Z$ such that $2\leq n < |m|$ and $2\leq n' < |m'|$. If $L(\BS(n,m))$ is stably isomorphic with $L(\BS(n',m'))$, then $\frac{n}{|m|}=\frac{n'}{|m'|}$.
\end{thmstar}

Note that Theorem \ref{thm.A} formally resembles, but is independent of, the results in \cite{Ki11} on orbit equivalence relations of essentially free ergodic probability measure preserving actions of Baumslag-Solitar groups, especially \cite[Proposition B.2 and Theorem 1.2]{Ki11}. It would be very interesting to find a framework that unifies both types of results.

We prove our Theorem \ref{thm.A} by associating a canonical equivalence relation to $L(\BS(n,m))$ and proving that it is of type III$_{n/|m|}$. More precisely,
assume that $(M,\tau)$ is a von Neumann algebra with separable predual, equipped with a faithful normal tracial state. Whenever $A \subset M$ is an abelian von Neumann subalgebra, the normalizer
$$\cN_M(A) := \{u \in \cU(M) \mid u A u^* = A \}$$
induces a group of trace preserving automorphisms of $A$. Writing $A = L^\infty(X,\mu)$ with $\mu$ being induced by $\tau_{|A}$, the corresponding orbit equivalence relation is a countable probability measure preserving (pmp) equivalence relation on $(X,\mu)$.

More generally, we can consider the set of partial isometries
\begin{equation}\label{myset}
\{u \in M \mid u^* u \;\;\text{and}\;\; uu^* \;\;\text{are projections in $A' \cap M$ and}\;\; u A u^* = A uu^* \, \} \; .
\end{equation}
Every such partial isometry induces a partial automorphism of $A$ and hence a partial automorphism of $(X,\mu)$. We denote by $\cR(A \subset M)$ the equivalence relation generated by all these partial automorphisms. When $A \subset M$ is maximal abelian, i.e.\ $A' \cap M = A$, then $\cR(A \subset M)$ coincides with the orbit equivalence relation induced by the normalizer $\cN_M(A)$. In particular, in that case the equivalence relation $\cR(A \subset M)$ preserves the probability measure $\mu$.

If however $A \subset M$ is not maximal abelian, the partial automorphisms of $A$ induced by the partial isometries in the set \eqref{myset} need not be trace preserving. So in general, $\cR(A \subset M)$ can be an equivalence relation of type III.

Our main technical result is Theorem \ref{thm.stable} below, roughly saying the following. If $A, B \subset M$ are abelian subalgebras such that $\cZ(A' \cap M) = A$ and $\cZ(B' \cap M) = B$, and if there exist intertwining bimodules $A \prec B$ and $B \prec A$ (in the sense of Popa, see \cite{Po03} and Theorem \ref{thm.intertwine} below), then the equivalence relations $\cR(A \subset M)$ and $\cR(B \subset M)$ must be stably isomorphic. In particular, their types must be the same.

In Section \ref{sec.class}, we apply this to $M = L(\BS(n,m))$ and $A$ equal to the abelian von Neumann subalgebra generated by the unitary $u_a$. We prove that $\cR(A \subset M)$ is the unique hyperfinite ergodic equivalence relation of type III$_{n / |m|}$.

The proof of Theorem \ref{thm.A} can then be outlined as follows. First we note that the von Neumann algebra $A' \cap M$ is nonamenable. Conversely if $Q \subset M$ is a nonamenable subalgebra, it was proven in \cite{CH08}, using spectral gap rigidity (\cite{Po06}) and the structure theory of amalgamated free product factors (\cite{IPP05}), that $Q' \cap M \prec A$. So, up to intertwining-by-bimodules, the position of $A$ inside $M$ is ``canonical''. Therefore a stable isomorphism between $L(\BS(n,m))$ and $L(\BS(n',m'))$ will preserve, up to intertwining-by-bimodules, these canonical abelian subalgebras. Hence their associated equivalence relations are stably isomorphic and, in particular, have the same type. This gives us the equality $n/|m| = n'/|m'|$.

\section{Preliminaries}\label{sec.pre}

We denote by $(M,\tau)$ a von Neumann algebra equipped with a faithful normal tracial state $\tau$. We always assume that $M$ has a \emph{separable predual}.
If $B$ is a von Neumann subalgebra of $(M,\tau)$, we denote by $E_B$ the unique trace preserving \emph{conditional expectation} of $M$ onto $B$.

Whenever $x \in M$ is a normal element, we denote by $\supp(x)$ its support, i.e.\ the smallest projection $p \in M$ that satisfies $xp = x$ (or equivalently, $px = x$).

Let $\cR$ be a countable nonsingular (i.e.\ measure class preserving) equivalence relation on a standard probability space $(X,\mu)$. We denote by $[[\cR]]$ the \emph{full pseudogroup} of $\cR$, i.e.\ the pseudogroup of all partial nonsingular automorphisms $\varphi$ of $X$ such that the graph of $\varphi$ is contained in $\cR$. We denote the domain of $\varphi$ by $\dom(\varphi)$ and its range by $\ran(\varphi)$. We denote by $[x]$ the equivalence class of $x\in X$.

Assume that also $\cR'$ is a countable nonsingular equivalence relation on the standard probability space $(X',\mu')$. The equivalence relations $\cR$ and $\cR'$ are called
\begin{itemize}
\item \emph{isomorphic}, if there exists a nonsingular isomorphism $\Delta:X\rightarrow X'$ such that $\Delta([x])=[\Delta(x)]$ for almost every $x \in X$~;
\item \emph{stably isomorphic}, if there exist Borel subsets $Z \subset X$ and $Z'\subset X'$ that meet almost every orbit and a nonsingular isomorphism $\Delta:Z \rightarrow Z'$ such that $\Delta([x]\cap Z)=[\Delta(x)]\cap Z'$ for almost every $x \in Z$.
\end{itemize}

\subsection{HNN extensions and Baumslag-Solitar groups}\label{subsec.HNN}

Let $G$ be a group, $H < G$ a subgroup and $\theta:H\rightarrow G$ an injective group homomorphism. The \emph{HNN extension} $\HNN(G,H,\theta)$ is defined as the group generated by $G$ and an additional element $t$ subject to the relation $\theta(h) = tht^{-1}$ for all $h \in H$. So,
$$\HNN(G,H,\theta) = \langle G,t \mid \theta(h)=tht^{-1} \;\;\text{for all} \;\; h\in H\rangle \; .$$

Elements of $\HNN(G,H,\theta)$ can be canonically written as ``reduced words'' using as letters the elements of $G$ and the letters $t^{\pm 1}$. More precisely, we have the following lemma.

\begin{lemma}[Britton's lemma, \cite{Br63}]\label{lem.brit}
Consider the expression $g = g_0 t^{n_1} g_1 t^{n_2} \cdots t^{n_k} g_k$ with $k \geq 0$, $g_0,g_k \in G$, $g_1,\ldots,g_{k-1} \in G-\{e\}$ and $n_1,\ldots,n_k \in \Z - \{0\}$. We call this expression reduced if the following two conditions hold:
\begin{itemize}
\item for every $i \in \{1,\ldots,k-1\}$ with $n_i > 0$ and $n_{i+1} < 0$, we have $g_i \not\in H$,
\item for every $i \in \{1,\ldots,k-1\}$ with $n_i < 0$ and $n_{i+1} > 0$, we have $g_i \not\in \theta(H)$.
\end{itemize}
If the above expression for $g$ is reduced, then $g \neq e$ in the group $\HNN(G,H,\theta)$, unless $k=0$ and $g_0 = e$. In particular, the natural homomorphism of $G$ to $\HNN(G,H,\theta)$ is injective.
\end{lemma}

Recall from the introduction that the Baumslag-Solitar group $\BS(n,m)$ is defined for all $n,m\in \Z - \{0\}$ as
$$\BS(n,m):=\langle a,b \mid ba^nb^{-1}=a^m\rangle \; .$$
It is one of the easiest examples of an HNN extension. We also recall from the introduction that the $\BS(n,m)$ with $2 \leq n < |m|$ form a complete list of all nonamenable icc Baumslag-Solitar groups up to isomorphism. Since we only want to consider the case where $L(\BS(n,m))$ is a nonamenable II$_1$ factor, we always assume that $2\leq n < |m|$.

\subsection{Hilbert bimodules and intertwining-by-bimodules}\label{subsec.inter}

If $M$ and $N$ are tracial von Neumann algebras, then a \emph{left $M$-module} is a Hilbert space $\cH$ endowed with a normal $*$-homomorphism $\pi:M\rightarrow \B(\cH)$. A \emph{right $N$-module} is a left $N\op$-module. An \emph{$M$-$N$-bimodule} is a Hilbert space $\cH$ endowed with commuting normal $*$-homomorphisms $\pi:M\rightarrow \B(\cH)$ and $\varphi:N\op\rightarrow \B(\cH)$. For $x\in M, y\in N$ and $\xi\in \cH$, we write $x\xi y$ instead of $\pi(x)\varphi(y\op)(\xi)$. We denote an $M$-$N$-bimodule $\cH$ by $\bim{M}{\cH}{N}$. We call an $M$-$N$-bimodule \emph{bifinite} if it is finitely generated both as a left Hilbert $M$-module and a right Hilbert $N$-module.

Let $A$ and $B$ be abelian von Neumann algebras. We denote by $\PIso(A,B)$ the set of all partial isomorphisms from $A$ to $B$, i.e.\ isomorphisms $\alpha:Aq\rightarrow Bp$, where $q\in A$ and $p\in B$ are projections. We write $\PAut(A)$ instead of $\PIso(A,A)$. Note that to every $\alpha\in \PIso(A,B)$ we can associate an $A$-$B$-bimodule $\bim{A}{\cH(\alpha)}{B}$ given by $\cH(\alpha)=L^2(Bp)$ and $a\xi b=\alpha(aq) \, \xi \, bp$. The composition of two partial isomorphisms is defined as follows: if $\alpha \in \PIso(B,C)$ and $\beta \in \PIso(A,B)$ are given by $\alpha : Bp \recht Cr$ and $\beta : Aq \recht Bp'$ for projections $q \in A$, $p,p' \in B$ and $r \in C$, then the composition $\alpha \circ \beta \in \PIso(A,C)$ is defined by $x \mapsto \alpha(\beta(x))$ for all $x \in A q \beta^{-1}(pp')$.

The following is a well known result.
\begin{lemma}\label{lem.finite}
Let $A$ and $B$ be abelian von Neumann algebras. Then every bifinite $A$-$B$-bimodule $\bim{A}{\cH}{B}$ is isomorphic to a direct sum of bimodules of the form $\bim{A}{\cH(\alpha)}{B}$ with $\alpha\in \PIso(A,B)$.
\end{lemma}

We finally recall Popa's \emph{intertwining-by-bimodules} theorem.

\begin{theorem}[{\cite[Theorem 2.1 and Corollary 2.3]{Po03}}]\label{thm.intertwine}
Let $(M,\tau)$ be a tracial von Neumann algebra and let $A,B\subset M$ be possibly nonunital von Neumann subalgebras. Denote their respective units by $1_A$ and $1_B$. The following three conditions are equivalent:
\begin{enumerate}
\item $1_A L^2(M) 1_B$ admits a nonzero $A$-$B$-subbimodule that is finitely generated as a right $B$-module.
\item There exist nonzero projections $p\in A$, $q\in B$, a normal unital $*$-homomorphism $\psi:pAp\rightarrow qBq$ and a nonzero partial isometry $v\in pMq$ such that $av=v\psi(a)$ for all $a\in pAp$.
\item There is no sequence of unitaries $u_n\in \cU(A)$ satisfying $||E_B(xu_ny^*)||_2\rightarrow 0$ for all $x,y\in 1_BM1_A$.
\end{enumerate}
If one of these equivalent conditions holds, we write $A \prec_M B$.
\end{theorem}

\subsection{Quasi-regularity}\label{subsec.quasi}

Let $(M,\tau)$ be a tracial von Neumann algebra and $N\subset M$ a von Neumann subalgebra. We denote by $\QN_M(N)$ the \emph{quasi-normalizer} of $N$ inside $M$, i.e.\ the unital $*$-algebra defined by
$$\Bigl\{a\in M \;\Big|\; \exists b_1,\ldots, b_k\in M,\exists d_1,\ldots,d_r\in M \;\;\text{such that}\;\; Na\subset \sum_{i=1}^k b_iN \;\;\textrm{and}\;\; aN\subset \sum_{j=1}^r Nd_j\Bigr\} \; .$$
We call $N\subset M$ \emph{quasi-regular} if $\QN_M(N)\dpr=M$.

If $A,B\subset M$ are abelian von Neumann subalgebras, we define $\qab$ as
$$\qab := \{v \in M \mid vv^* \in A' \cap M \; , \; v^* v \in B' \cap M \;\;\text{and}\;\; A v = v B \} \; .$$
Whenever $v \in \qab$, we define $q_v=\supp(E_A(vv^*))$ and $p_v=\supp(E_B(v^*v))$, and we denote by $\alpha_v : Aq_v\rightarrow Bp_v$ the unique $*$-isomorphism satisfying $a v = v \alpha_v(a)$ for all $a \in A q_v$.

Note that the set $\qab$ can be $\{0\}$. In Lemma \ref{lem.decomp}, we will see that $\qab\neq \{0\}$ if and only if there exists a bifinite $A$-$B$-subbimodule $\bim{A}{\cH}{B}$ of $\bim{A}{L^2(M)}{B}$.

We denote $\Qu_M(A,A)$ by $\qam$.

\begin{lemma}\label{lem.decomp}
Let $(M,\tau)$ be a tracial von Neumann algebra and $A,B\subset M$ abelian von Neumann subalgebras. Then the following statements hold.
\begin{enumerate}
\item If $\alpha\in \PIso(A,B)$ and if $\theta:\bim{A}{\cH(\alpha)}{B}\rightarrow\bim{A}{L^2(M)}{B}$ is an $A$-$B$-bimodular isometry, then there exists a partial isometry $v\in \qab$ such that $\alpha=\alpha_v$ and such that
\[\theta(\cH(\alpha))\subset \overline{v(B'\cap M)}^{||\cdot||_2}\subset \overline{\lspan}^{||\cdot||_2}\qab \; .\]
\item Every bifinite $A$-$B$-subbimodule $\bim{A}{\cH}{B}$ of $\bim{A}{L^2(M)}{B}$ is contained in $\overline{\lspan}^{||\cdot||_2}\qab$.
\item $\qam\dpr=\QN_M(A)\dpr$.
\item We have $\qab \neq \{0\}$ if and only if $\bim{A}{L^2(M)}{B}$ admits a nonzero bifinite $A$-$B$-subbimodule.
\end{enumerate}
\end{lemma}
\begin{proof}
1. Let $\alpha:Aq\rightarrow Bp$ be an element of $\PIso(A,B)$. Define $\xi:=\theta(p)\in L^2(M)$ and let $\xi=v|\xi|$ be its polar decomposition.
For all $a \in A$, we have $a\xi=\xi\alpha(a)$ and hence, $av=v\alpha(a)$. Furthermore $p=\supp(E_B(v^*v))$ and $q=\alpha^{-1}(p)=\supp(E_A(vv^*))$. So we find that $v\in  \qab$ and $\alpha=\alpha_v$. Because $|\xi|\in L^2(B'\cap M)$, we have that $\xi=v|\xi|$ is an element of $\overline{v(B'\cap M)}^{||\cdot||_2}$. Since $p$ generates $\bim{A}{\cH(\alpha)}{B}$ as a right Hilbert $B$-module, we have proven the first inclusion $\theta(\cH(\alpha))\subset \overline{v(B'\cap M)}^{||\cdot||_2}$. Since $v \in \qab$, also $v(B' \cap M) \subset \qab$ and the second inclusion in statement~1 is proven as well.

2. Let $\bim{A}{\cH}{B}$ be a bifinite $A$-$B$-subbimodule of $\bim{A}{L^2(M)}{B}$. By Lemma \ref{lem.finite}, $\bim{A}{\cH}{B}$ is isomorphic to a direct sum of bimodules of the form $\bim{A}{\cH(\alpha_i)}{B}$ with $\alpha_i \in \PIso(A,B)$. Using statement~1 of the lemma, we find that $\cH$ is generated by subspaces of $\overline{\lspan}^{||\cdot||_2} \qab$. This proves statement~2.

3. By definition, we have $\qam\dpr\subset \QN_M(A)\dpr$. On the other hand, $\bim{A}{L^2(\QN_M(A)\dpr)}{A}$ is a direct sum of bifinite $A$-$A$-subbimodules of $\bim{A}{L^2(M)}{A}$. So by statement~2, we have that $L^2(\QN_M(A)\dpr)\subset \overline{\lspan}^{||\cdot||_2}(\qam)$. Therefore we conclude that $\QN_M(A)\dpr= \qam\dpr$.

Finally, 4 is an immediate consequence of 2.
\end{proof}

We end this subsection with the following lemma, clarifying why later, we will consider abelian subalgebras $A \subset M$ satisfying $\cZ(A' \cap M) = A$. Note that since $A$ is abelian, the condition $\cZ(A' \cap M) = A$ is equivalent with the ``bicommutant'' property $(A' \cap M)' \cap M = A$. Also note that the composition of two partial isomorphisms was defined before Lemma \ref{lem.finite}.

\begin{lemma}\label{lem.product}
Let $(M,\tau)$ be a tracial von Neumann algebra and $A,B,C \subset M$ abelian von Neumann subalgebras.
If $v\in \Qu_M(A,B)$, $w\in \Qu_M(B,C)$ and if $\cZ(B'\cap M)=B$, then there exists an element $u\in \Qu_M(A,C)$ such that $\alpha_w\circ \alpha_v = \alpha_u$.
\end{lemma}
\begin{proof}
Choose $v\in \Qu_M(A,B)$ and $w\in \Qu_M(B,C)$. Note that $vbw\in \Qu_M(A,C)$ for every $b\in B'\cap M$ and $\alpha_{vbw}=\alpha_w\circ {\alpha_v}_{| Aq_{vbw}}$. We claim that $\bigvee_{b\in B'\cap M} q_{vbw} =\alpha_v^{-1}(q_wp_v)$.

Denote $\alpha_v^{-1}(q_wp_v) - \bigvee_{b\in B'\cap M} q_{vbw}$ by $r$. We need to prove that $r$ is zero. Since $(B'\cap M)'\cap M=B$, we have that for every $x\in M$, the projection $\supp(E_B(xx^*))$ equals the projection of $L^2(M)$ onto the closed linear span of $(B'\cap M)xM\subset L^2(M)$. Since $w^*bv^*r=0$ for every $b\in B'\cap M$, it follows that $w^*q_{v^*r}=0$. Therefore $q_w$ is orthogonal to $q_{v^*r}$. Because $q_{v^*r}=\alpha_v(r)$ and $\alpha_v(r)\leq q_w$, it follows that $\alpha_v(r) = 0$. Hence $r=0$ and our claim that $\bigvee_{b\in B'\cap M} q_{vbw} =\alpha_v^{-1}(q_wp_v)$ is proven.

By cutting down with appropriate projections, we find $b_n \in B' \cap M$ such that the projections $q_{v b_n w}$ are orthogonal and sum up to $\alpha_v^{-1}(q_wp_v)$. In particular, the left supports, resp.\ right supports, of the elements $v b_n w$ are orthogonal. So we can define $u = \sum_n v b_n w$. It follows that $u \in \Qu_M(A,C)$ and $\alpha_w\circ\alpha_v=\alpha_u$.
\end{proof}

\subsection{The type of an ergodic nonsingular countable equivalence relation}\label{subsec.type}

Let $\cR$ be a nonsingular ergodic countable Borel equivalence relation on a standard probability space $(X,\mu)$. Using the map $\pi : \cR \recht X : \pi(x,y) = x$, we define the measure $\mu^{(1)}$ on $\cR$ given by
$$\mu^{(1)}(U) = \int_X \#(U \cap \pi^{-1}(x)) \; d\mu(x) \quad\text{for all Borel sets}\;\; U \subset \cR \; .$$
We define $\cR^{(2)}:=\{(x,y,z)\in X^3 \mid (x,y),(y,z)\in\cR\}$. Similarly, using the map $\rho : \cR^{(2)} \recht X : \rho(x,y,z) = x$, we define the measure $\mu^{(2)}$ on $\cR^{(2)}$ given by
$$\mu^{(2)}(V) = \int_X \#(V \cap \rho^{-1}(x)) \; d\mu(x) \quad\text{for all Borel sets}\;\; V \subset \cR^{(2)} \; .$$
The \emph{Radon-Nikodym 1-cocycle} of $\cR$ is the $\mu^{(1)}$-a.e.\ uniquely defined Borel map $\omega:\cR\rightarrow \R$ such that
\[\omega(\varphi(x),x)=\log\Bigl(\frac{d \mu\circ\varphi}{d \mu}(x)\Bigr) \quad\text{for all}\;\; \varphi\in [[\cR]] \;\;\textrm{and almost every}\;\; x\in \dom \varphi \; .\]
Note that $\omega$ satisfies the $1$-cocycle relation $\omega(x,z)=\omega(x,y)+\omega(y,z)$ for $\mu^{(2)}$-a.e.\ $(x,y,z)\in \cR^{(2)}$. One then defines the \emph{Maharam extension} $\widetilde{\cR}$ of $\cR$ as the equivalence relation on $(X\times \R, \mu \times \exp(-t)dt)$ defined by
\[(x,t)\sim (y,s) \quad\text{if and only if}\quad (x,y)\in \cR \;\;\text{and}\;\; t-s=\omega(x,y).\]
Note that $\mu \times \exp(-t)dt$ is an infinite invariant measure for $\widetilde{\cR}$.
Denote the von Neumann algebra of all $\widetilde{\cR}$-invariant functions in $L^\infty(X\times \R)$ by $L^\infty(X\times \R)^{\widetilde{\cR}}$. Since $\cR$ was assumed to be ergodic, one can easily check that the action of $\R$ on $L^\infty(X\times \R)^{\widetilde{\cR}}$ given by translation of the second variable, is also ergodic. Depending on how this action of $\R$ looks like, we define as follows the type of $\cR$.
\begin{itemize}
\item I or II, if the action is conjugate with $\R\actson \R$~;
\item III$_\lambda$ $(0<\lambda<1)$, if the action is conjugate with $\R\actson \R/\Z \log(\lambda)$~;
\item III$_1$, if the action is on one point~;
\item III$_0$, if the action is properly ergodic, i.e.\ is ergodic and has orbits of measure zero.
\end{itemize}

\begin{remark}
Denote by $L(\cR)$ the von Neumann algebra associated with $\cR$. Denote by $\vphi$ the normal semifinite faithful state on $L(\cR)$ that is induced by $\mu$. Finally denote by $(\si_t^\vphi)_{t \in \R}$ its modular automorphism group. There is a canonical identification $L(\widetilde{\cR})\cong L(\cR) \rtimes_{\sigma^{\varphi}}\R$.
Under this identification, the dual action of $\R$ on $L(\cR) \rtimes_{\sigma^{\varphi}}\R$ corresponds to the action of $\R$ on $L(\widetilde{\cR})$ that we defined above. Also, the center of $L(\cR) \rtimes_{\sigma^{\varphi}}\R$ corresponds to $L^\infty(X\times \R)^{\widetilde{\cR}}$. Altogether it follows that the type of the equivalence relation $\cR$ coincides with the type of the factor $L(\cR)$.
\end{remark}

\begin{lemma}\label{lem.type}
Let $\cR$ be a nonsingular ergodic countable Borel equivalence relation on the standard probability space $(X,\mu)$. Denote by $\omega$ its Radon-Nikodym 1-cocycle. If the essential image $\Ima(\omega)$ of $\omega$ equals $\log(\lambda)\Z$ for some $0<\lambda<1$ and if the kernel $\Ker(\omega)$ of $\omega$ is an ergodic equivalence relation, then $\cR$ is of type III$_\lambda$.
\end{lemma}
\begin{proof}
Since $\Ker(\omega)$ is an ergodic equivalence relation on $(X,\mu)$, we have
$$L^\infty(X \times \R)^{\cRtil} \subset L^\infty(X \times \R)^{\Ker(\om)} = 1 \ot L^\infty(\R) \; .$$
For a given $F \in L^\infty(\R)$, we have that $1 \ot F$ is $\cRtil$-invariant if and only if $F$ is invariant under translation by the essential image of $\om$. So,
$$L^\infty(X \times \R)^{\cRtil} = 1 \ot L^\infty(\R / \log(\lambda) \Z) \; .$$
\end{proof}

\section{Equivalence relations associated to subalgebras that are abelian, but not maximal abelian}\label{sec.rel}

Throughout this section, we fix a tracial von Neumann algebra $(M,\tau)$ with separable predual. We also fix an abelian von Neumann subalgebra $A \subset M$ satisfying $\cZ(A' \cap M) = A$. Choose a standard probability space $(X,\mu)$ such that $A = L^\infty(X,\mu)$. For every nonsingular partial automorphism $\vphi$ of $(X,\mu)$, we denote by $\al_\vphi$ the corresponding partial automorphism of $A$.

We first prove that $\Qu_M(A)$ induces a nonsingular countable Borel equivalence relation $\cR(A \subset M)$ on $(X,\mu)$. For this, we introduce the notation
\begin{equation}\label{not.withG}
\cG(A \subset M) := \{\al_v \mid v \in \Qu_M(A) \} \; .
\end{equation}

\begin{proposition}\label{prop.our-equiv}
There exists a nonsingular countable Borel equivalence relation $\cR$ on $(X,\mu)$ with the following property: a nonsingular partial automorphism $\vphi$ of $X$ satisfies $\al_\vphi \in \cG(A \subset M)$ if and only if $(x,\vphi(x)) \in \cR$ for a.e.\ $x \in \dom(\vphi)$.

Moreover, $\cR$ is essentially unique: if a nonsingular countable Borel equivalence relation $\cR'$ on $(X,\mu)$ satisfies the same property, then there exists a Borel subset $X_0 \subset X$ with $\mu(X-X_0) = 0$ and $\cR_{|X_0} = \cR'_{|X_0}$.

We denote $\cR(A \subset M) := \cR$. The equivalence relation $\cR(A \subset M)$ is ergodic if and only if $\QN_M(A)\dpr$ is a factor.
\end{proposition}

Before proving Proposition \ref{prop.our-equiv}, we introduce some terminology and a lemma. To every $\al \in \PAut(A)$ are associated the support projections $q_\al,p_\al \in A$ such that $\al : A q_\al \recht A p_\al$ is a $*$-isomorphism. Assume that $\al \in \PAut(A)$ and $\cF \subset \PAut(A)$. We say that $\al$ is a \emph{gluing} of elements in $\cF$, if there exists a sequence of elements $\al_n \in \cF$ and projections $q_n \in A$ such that $q_\al = \sum_n q_n$ and such that $q_n \leq q_{\al_n}$ and $\al_{|Aq_n} = {\al_n}_{|A q_n}$ for all $n$.

\begin{lemma}\label{lem.glue}
Let $\cJ \subset \Qu_M(A)$ and $v \in \Qu_M(A)$ such that $v \in \overline{\lspan}^{||\cdot||_2} \cJ$. Then $\al_v$ is a gluing of elements in $\{\al_w \mid w \in \cJ\}$.
\end{lemma}
\begin{proof}
By a standard maximality argument, it suffices to prove that for every nonzero projection $q \in A q_v$, there exists a nonzero subprojection $q_0 \in A q$ and a $w \in \cJ$ such that $q_0 \leq q_w$ and ${\al_v}_{|A q_0} = {\al_w}_{|A q_0}$.

So fix a nonzero projection $q \in A q_v$. It follows that $q E_A(vv^*) \neq 0$. Since $v \in \overline{\lspan}^{||\cdot||_2} \cJ$, we can pick a $w \in \cJ$ such that $q E_A(v w^*) \neq 0$. Define $q_1 :=  \supp(E_A(v w^*))$ and note that $q_1 \in q_v A q_w = A q_v q_w$. Also note that $q q_1 \neq 0$. For all $a \in A$, we have
$$\al_v^{-1}(a p_v) \, v w^* = v \, a \, w^* = v w^* \, \al_{w}^{-1}(a p_w) \; .$$
Applying the conditional expectation onto $A$ and using that $A$ is abelian, we find that
$$\al_v^{-1}(a p_v) \, q_1 = \al_w^{-1}(a p_w) \, q_1 \quad\text{for all}\;\; a \in A \; .$$
This means that ${\al_v}_{|A q_1} = {\al_w}_{|A q_1}$. We put $q_0 := q q_1$. We already showed that $q_0 \neq 0$. Since $q_0 \leq q_1$, we have that ${\al_v}_{|A q_0} = {\al_w}_{|A q_0}$.
\end{proof}

\begin{proof}[Proof of Proposition \ref{prop.our-equiv}]
We say that a subpseudogroup $\cG \subset \PAut(A)$ is \emph{of countable type} if there exists a countable subset $\cJ \subset \cG$ such that every $\al \in \cG$ is a gluing of elements in $\cJ$. To prove the first part of the proposition, we must show that $\cG(A \subset M)$ is a subpseudogroup of countable type of $\PAut(A)$. From Lemma \ref{lem.product}, it follows that $\cG(A \subset M)$ is indeed a subpseudogroup. Since $M$ has a separable predual, we can choose a countable $\|\cdot\|_2$-dense subset $\cJ \subset \Qu_M(A)$. By Lemma \ref{lem.glue}, every $\al \in \cG(A \subset M)$ is a gluing of elements in $\{\al_w \mid w \in \cJ\}$. Hence $\cG(A \subset M)$ is of countable type. So the first part of the proposition is proven and we can essentially uniquely define the nonsingular countable Borel equivalence relation $\cR$ on $(X,\mu)$.

Since $A' \cap M \subset \QN_M(A)\dpr$ and since we assumed that $(A' \cap M)' \cap M = A$, the center of $\QN_M(A)\dpr$ is a subalgebra of $A$. By Lemma \ref{lem.decomp}.3, we have $\QN_M(A)\dpr = \Qu_M(A)\dpr$. Therefore,
$$\cZ(\QN_M(A)\dpr) = \{a \in A \mid a v = v a \;\;\text{for all}\;\; v \in \Qu_M(A) \} \; .$$
The right hand side equals $A^\cR$, the subalgebra of $\cR$-invariant functions in $A$. So $\cR$ is ergodic if and only if $\QN_M(A)\dpr$ is a factor.
\end{proof}

For our application, the following theorem is crucial. It says that $\cR(A \subset M)$ remains the same, up to stable isomorphism, if we replace $A$ by an abelian subalgebra $B$ that has a mutual intertwining bimodule into $A$.

\begin{theorem}\label{thm.stable}
Let $M$ be a II$_1$ factor with separable predual. Let $A,B\subset M$ be abelian, quasi-regular von Neumann subalgebras satisfying $\cZ(A'\cap M)=A$ and $\cZ(B'\cap M)=B$. If $A\prec_M B$ and $B\prec_M A$, then the equivalence relations $\ram$ and $\cR(B\subset M)$ are stably isomorphic.
\end{theorem}
\begin{proof}
Since $A,B$ are quasi-regular and since $A\prec_M B$ as well as $B\prec_M A$, there exists a nonzero bifinite $A$-$B$-subbimodule of $L^2(M)$.
So by Lemma \ref{lem.decomp}.4, there exists a nonzero element $v \in \qab$ with corresponding $\al_v \in \PAut(A,B)$. Using the notation in \eqref{not.withG} and using Lemma \ref{lem.product}, we find that
\begin{align*}
\al_v \circ \beta \circ \al_v^{-1} \in \cG(B \subset M) \quad & \text{for all}\;\; \beta \in \cG(A \subset M) \quad\text{and}\\
\al_v^{-1} \circ \gamma \circ \al_v \in \cG(A \subset M) \quad & \text{for all}\;\; \gamma \in \cG(B \subset M) \; .
\end{align*}
So $\al_v$ implements a stable isomorphism between $\cR(A \subset M)$ and $\cR(B \subset M)$.
\end{proof}

The following lemma will allow us to easily compute $\cR(A \subset M)$ in concrete examples.

\begin{lemma}\label{lem.voortbr}
Let $(M,\tau)$ be a tracial von Neumann algebra and $A \subset M$ an abelian von Neumann subalgebra satisfying $\cZ(A' \cap M) = A$.
Let $\cF\subset M$ be a subset such that
\begin{itemize}
\item $M = (\cF\cup \cF^*\cup (A'\cap M))\dpr$,
\item as an $A$-$A$-bimodule, $\overline{\lspan}^{||\cdot||_2} A\cF A$ is isomorphic to a direct sum of bimodules of the form $\bim{A}{\cH(\alpha_n)}{A}$ with $\alpha_n\in \PAut(A)$.
\end{itemize}
Choose nonsingular partial automorphisms $\vphi_n$ of $(X,\mu)$ such that $\al_n = \al_{\vphi_n}$ for all $n$. Up to measure zero, $\ram$ is generated by the graphs of the partial automorphisms $\vphi_n$.
\end{lemma}
\begin{proof}
We again use the notation \eqref{not.withG}. By Lemma \ref{lem.decomp}.1, we find $v_n \in \Qu_M(A)$ such that $\al_n = \al_{v_n}$ and
\begin{equation}\label{eq.aneq1}
\overline{\lspan}^{||\cdot||_2} A\cF A \subset \overline{\lspan}^{||\cdot||_2} \{ v_n (A' \cap M) \mid n \in \N\} \; .
\end{equation}
In particular, we have $\al_n \in \cG(A \subset M)$. Choose nonsingular partial automorphisms $\vphi_n$ of $(X,\mu)$ such that $\al_n = \al_{\vphi_n}$ for all $n$.

Denote by $\cR$ the smallest (up to measure zero) equivalence relation on $(X,\mu)$ that contains the graphs of all the partial automorphisms $\vphi_n$. By the previous paragraph, we know that $\cR$ is a subequivalence relation of $\cR(A \subset M)$. Denote by $\cJ$ the set of all products of elements in
$$\{v_n \mid n \in \N\} \cup \{v_n^* \mid n \in \N\} \cup (A' \cap M) \; .$$
By construction, the graph of every $\al_w$, $w \in \cJ$, belongs to $\cR$. Combining our assumption that $M = (\cF\cup \cF^*\cup (A'\cap M))\dpr$ with \eqref{eq.aneq1}, it follows that $\lspan \cJ$ is $\|\cdot\|_2$-dense in $L^2(M)$. By Lemma \ref{lem.glue}, every $\al \in \cG(A \subset M)$ is a gluing of elements in $\{\al_w \mid w \in \cJ\}$. So the graph of every $\al \in \cG(A \subset M)$ belongs to $\cR$ a.e. Hence $\cR$ equals $\cR(A \subset M)$ almost everywhere.
\end{proof}

We finally note in the following proposition that every nonsingular countable Borel equivalence relation $\cR$ arises as $\cR(A \subset M)$.

\begin{proposition}
Let $\cR$ be a nonsingular countable Borel equivalence relation. Then there exists a quasi-regular inclusion of an abelian von Neumann algebra $A$ in a tracial von Neumann algebra $(M,\tau)$ satisfying $\cZ(A'\cap M)=A$ and such that $\cR\cong \cR(A\subset M)$.
\end{proposition}
\begin{proof}
Let $\cR$ be a nonsingular countable Borel equivalence relation on a standard probability space $(X,\mu)$. Denote by $(P,\Tr)$ the unique hyperfinite II$_\infty$ factor and choose a trace-scaling action $(\alpha_t)_{t \in \R}$ of $\R$ on $P$. This means that $\Tr\circ \alpha_t=e^{-t}\Tr$. The corresponding action of $\R$ on $L^2(P)$ will also be denoted by $(\alpha_t)$. We denote by $\omega : \cR \recht \R$ the Radon-Nikodym $1$-cocycle of $\cR$ (see Section \ref{subsec.type}).

In the same way as with the Maharam extension of a nonsingular group action, the equivalence relation $\cR$ admits a natural trace preserving action on $L^\infty(X) \ovt P$. We denote by $(\cM,\Tr)$ the crossed product. For completeness, we recall the construction of $(\cM,\Tr)$. To every $\vphi\in [[\cR]]$, we associate the operator $W_\vphi$ on $L^2(\cR,L^2(P))$ given by
$$(W_\vphi\xi)(x,y)=\begin{cases} \alpha_{\omega(x,\vphi^{-1}(x))}(\xi(\vphi^{-1}(x),y)) &\quad\text{if}\;\; x\in \dom(\vphi^{-1}) \; , \\ 0 &\quad\text{otherwise,} \end{cases}$$
for every $\xi\in L^2(\cR,L^2(P))$. One checks that $W_\vphi W_\psi = W_{\vphi\circ\psi}$ and $W_\vphi^*=W_{\vphi^{-1}}$.

We represent $L^\infty(X) \ovt P = L^\infty(X,P)$ on $L^2(\cR,L^2(P))$ by
$$(F\xi)(x,y)=F(x)\xi(x,y) \;\;\text{for all}\;\; \xi\in L^2(\cR,L^2(P)) \;\;\text{and}\;\; F\in L^\infty(X,P) \; .$$
Note that the partial isometries $W_\vphi$, $\vphi \in [[\cR]]$, normalize $L^\infty(X,P)$ and that
$$(W_\vphi^*FW_\vphi)(x)=\begin{cases} \alpha_{\omega(x,\vphi(x))}(F(\vphi(x))) &\quad\text{if}\;\; x\in \dom\vphi \\ 0 &\quad\text{otherwise.} \end{cases}$$

Define $\cM$ as the von Neumann algebra generated by $L^\infty(X,P)$ and the partial isometries $W_\vphi$, $\vphi \in [[\cR]]$. Denoting by $\Delta \subset \cR$ the diagonal subset, the orthogonal projection onto $L^2(\Delta,L^2(P))$ implements a normal faithful conditional expectation $E : \cM \recht L^\infty(X) \ovt P$ satisfying
$$E(W_\vphi)=\chi_{\{x\mid \vphi(x)=x\}} \ot 1 \;\;\text{for all}\;\; \vphi\in [[\cR]] \; .$$
The formula $\Tr := (\mu \otimes \Tr) \circ E$ defines a normal semifinite faithful trace on $\cM$.
%Note that $\Tr(FW_\vphi):=\int_{\{x\mid \vphi(x)=x\}}\Tr(F(x)) d\mu(x)$.

Fix a nonzero projection $q\in P$ with $\Tr(q)=1$. Define the projection $p \in L^\infty(X) \ovt P$ given by $p = 1 \ot q$. Write $A:=L^\infty(X)p$ and $M:=p\cM p$. Then $A$ is a quasi-regular abelian von Neumann subalgebra of $M$ and the restriction of $\Tr$ to $M$ gives a normal faithful tracial state $\tau$ on $M$. The relative commutant $L^\infty(X)' \cap \cM$ equals $L^\infty(X) \ovt P$. Since $P$ is a factor, it follows that $\cZ(A'\cap M)=A$.

We finally prove that $\cR\cong \cR(A\subset M)$.
Write $\cR=\bigcup_k \graph(\vphi_k)$, with $\vphi_k\in [\cR]$. Then $\vphi_k$ induces an automorphism of $L^\infty(X)$ and hence of $A = L^\infty(X) p$ that we denote by $\beta_k \in \Aut(A)$. Since $P$ is a II$_\infty$ factor and $q \in P$ is a finite projection, we can choose partial isometries $w_n \in P$ such that $\sum_n w_n^* w_n = 1$ and $w_n w_n^* = q$ for all $n$. Define the elements
$$v_{n,k} := (1 \ot w_n) W_{\vphi_k} p \; .$$
All $v_{n,k}$ belong to $\qam$ and $\alpha_{v_{n,k}}$ equals the restriction of $\beta_k$ to $A p_{n,k}$ for projections $p_{n,k} \in A$. Since the sum of all $w_n^* w_n$ equals $1$, we also have that $\bigvee_n p_{n,k} = p$. Therefore the graphs of the partial automorphisms $\alpha_{v_{n,k}}$ generate an equivalence relation that is isomorphic with $\cR$. To conclude the proof, we put $\cF := \{v_{n,k} \mid n,k \in \N\}$ and observe that $M = (\cF \cup \cF^* \cup (A' \cap M))\dpr$.
By Lemma \ref{lem.voortbr}, the equivalence relation $\cR(A \subset M)$ is generated by the graphs of the partial automorphisms $\alpha_{v_{n,k}}$.
\end{proof}

\section{\boldmath Proof of Theorem \ref{thm.A}}\label{sec.class}

Throughout this section, we assume that $n$ and $m$ are integers satisfying $2\leq n<|m|$. As explained in the introduction, the corresponding groups $\BS(n,m)$ form a complete list of the nonamenable icc Baumslag-Solitar groups up to isomorphism.

Throughout this section, we write $M = L(\BS(n,m))$ and $A=\{u_a , u_a^*\}\dpr$. We start with the following observation.

\begin{proposition}\label{prop.voorw}
We have that $A \subset M$ is a quasi-regular abelian von Neumann subalgebra satisfying $\cZ(A' \cap M) = A$. Moreover, $A' \cap M$ has no amenable direct summand.
\end{proposition}
\begin{proof}
It is clear that $A \subset M$ is a quasi-regular abelian von Neumann subalgebra, because the element $a \in \BS(n,m)$ generates an almost normal abelian subgroup of $\BS(n,m)$~: for every $g \in \BS(n,m)$, the group $g a^\Z g^{-1} \cap a^\Z$ has finite index in $a^\Z$.

To prove that $\cZ(A'\cap M)=A$, we define the finite index subalgebra $A_0:=\{u_a^n , u_a^{-n}\}\dpr$ of $A$. We will first prove that $\cZ(A_0'\cap M)=A_0$. Afterwards we will show that this implies that $\cZ(A'\cap M)=A$.

Define $G := \langle a^\Z,b^{-1}a^\Z b\rangle \subset \BS(n,m)$. Then $L(G)$ is a subalgebra of $A_0'\cap M$. So $\cZ(A_0'\cap M)\subset L(G)'\cap M$. Using Lemma \ref{lem.brit}, one can easily see that $\{g\gamma g^{-1}\mid g\in G\}$ is an infinite set for every $\gamma \in \BS(n,m)-a^{n\Z}$. Therefore $L(G)'\cap M\subset A_0$. This shows that $\cZ(A_0'\cap M)\subset A_0$. Since the converse inclusion is obvious, we find that $A_0 = \cZ(A_0'\cap M)$.

Since $A_0\subset A$ has finite index, there exist orthogonal projections $p_j\in A$ such that $Ap_j=A_0p_j$ and $\sum_j p_j=\dblone$. But then
\begin{align*}
\cZ(A'\cap M)p_j &= \cZ((A'\cap M)p_j)= \cZ((Ap_j)'\cap p_jMp_j)\\
&= \cZ((A_0p_j)'\cap p_jMp_j)=\cZ(p_j(A_0'\cap M)p_j)\\
&= \cZ(A_0'\cap M)p_j = A_0p_j \subset A
\end{align*}
Therefore $\cZ(A'\cap M)\subset A$. The converse inclusion being obvious, we have proven that $A = \cZ(A' \cap M)$.

Using Lemma \ref{lem.brit}, it follows that $G$ is an amalgamated free product of two copies of $\Z$ over a copy of $\Z$ embedded as $n\Z$ and $m \Z$ respectively. In particular, $G$ is nonamenable and $L(G)$ has no amenable direct summand. Since $L(G) \subset A_0' \cap M$, it follows that $A_0' \cap M$ has no amenable direct summand either. As above, we have that
$$(A' \cap M) p_j = p_j(A_0' \cap M) p_j$$
for all $j$. Hence $A' \cap M$ has no amenable direct summand.
\end{proof}

We now identify the associated countable equivalence relation $\cR(A \subset M)$.

\begin{proposition}\label{prop.type}
The equivalence relation $\ram$ is isomorphic with the unique hyperfinite ergodic countable equivalence relation of type III$_{n/|m|}$.
\end{proposition}

\begin{proof}
Let $k$ be the greatest common divisor of $n$ and $|m|$. Write $n = n_0 k$ and $m = m_0 k$. By our assumptions on $n$ and $m$, we have that $1 \leq n_0 < |m_0|$. Define the countable Borel equivalence relation $\rnm$ on the circle $\mathbb{T}$ given by
$$\rnm:= \bigl\{(y,z) \in \T \times \T  \bigm| \exists a ,b\in \N \;\;\textrm{such that}\;\; a+b > 0 \;\;\text{and}\;\; y^{(n_0^{a} m_0^{b} k)}=z^{(m_0^{a}n_0^{b} k)}\bigr\} \; .$$
Equip $\T$ with its Lebesgue measure $\lambda$ and note that $\rnm$ is a nonsingular countable Borel equivalence relation on $(\T,\lambda)$.

Define $\cR_0 :=\{(y,z)\in \T \times \T \mid y^{m}=z^{n}\}$. Note that $\cR_0 \subset \rnm$ and that $\rnm$ is the smallest equivalence relation containing $\cR_0$. Define $\pi : \cR_0 \recht \T : \pi(y,z) = y^m$. Note that $\pi$ is $n|m|$-to-$1$. Define the probability measure $\mu$ on $\cR_0$ given by
\[\mu(U)=\frac{1}{n|m|}\int_{\mathbb{T}}\# \bigl(U \cap \pi^{-1}(\{x\}) \bigr) \; d\lambda(x) \; .\]
For all $k,l \in \Z$, we define the function $P_{k,l} : \cR_0 \recht \T : P_{k,l}(y,z) = y^k z^l$. A direct computation yields a unique unitary
$$T:L^2(\cR_0,\mu)\rightarrow \overline{\lspan}^{||\cdot||_2} A u_b A : P_{k,l} \mapsto u_a^k u_b u_a^l \; .$$
We turn $L^2(\cR_0,\mu)$ into an $L^\infty(\mathbb{T})$-$L^\infty(\mathbb{T})$-bimodule by the formula
$$(F \cdot \xi \cdot F')(y,z) = F(y) \, \xi(y,z) \, F'(z) \; .$$
Under the natural identification of $L^\infty(\T)$ and $A$, the unitary $T$ is $A$-$A$-bimodular.

By construction $\bim{A}{L^2(\cR_0,\mu)}{A}$ is isomorphic with a direct sum of bimodules of the form $\bim{A}{\cH(\al_j)}{A}$ where the union of the graphs of the partial automorphisms $\al_j$ equals $\cR_0$ and hence generates the equivalence relation $\rnm$. Applying Lemma \ref{lem.voortbr} to $\cF = \{u_b\}$, we conclude that $\ram \cong \rnm$ up to measure zero.

As in Section \ref{subsec.type}, denote by $\om : \rnm \recht \R$ the Radon-Nikodym $1$-cocycle. Denote by $\Lambda \subset \T$ the subgroup given by
$$\Lambda := \Bigl\{\exp\Bigl( \frac{2\pi i s}{(n_0 m_0)^b} \Bigr) \; \Big| \; s \in \Z ,b \in \N \Bigr\} \; .$$
For every $z \in \Lambda$, we denote by $\al_z : \T \recht \T$ the rotation $\al_z(y) = zy$. We have $\graph \al_z \subset \rnm$ for all $z \in \Lambda$. Since all $\al_z$ are measure preserving, we actually have $\graph \al_z \subset \Ker(\om)$. Since $\Lambda \subset \T$ is a dense subgroup, it follows that $\Ker(\om)$ is an ergodic equivalence relation. In particular, $\rnm$ is ergodic. A direct computation shows that $\om(y,z) = \log(n/|m|)$ for all $(y,z) \in \cR_0$. Since $\cR_0$ generates the equivalence relation $\rnm$, it follows that the essential image of $\om$ equals $\log(n/|m|) \Z$. Using Lemma \ref{lem.type}, we conclude that $\rnm$ is of type III$_{n/|m|}$. By construction, $\rnm$ is amenable and hence, hyperfinite.
\end{proof}

We are now ready to prove our main theorem.

\begin{proof}[Proof of Theorem \ref{thm.A}]
Fix for $i = 1,2$, integers $n_i, m_i \in \Z$ with $2 \leq n_i < |m_i|$. Put $M_i = L(\BS(n_i,m_i))$ and denote by $A_i \subset M_i$ the abelian von Neumann subalgebra generated by $u_a$, where $a \in \BS(n_i,m_i)$ is the first canonical generator. Assume that $M_1$ and $M_2$ are stably isomorphic. We must prove that
\begin{equation}\label{desire}
\frac{n_1}{|m_1|} = \frac{n_2}{|m_2|} \; .
\end{equation}
Interchanging if necessary the roles of $M_1$ and $M_2$, we can take a nonzero projection $p_1 \in A_1$ and a $*$-isomorphism $\al : p_1 M_1 p_1 \recht M_2$.

We claim that inside $M_2$, we have $\al(A_1 p_1) \prec A_2$. From Proposition \ref{prop.voorw}, we know that
$$P := \al(A_1 p_1)' \cap M_2 = \al((A_1' \cap M_1)p_1)$$
has no amenable direct summand. By Proposition 3.1 in \cite{Ue07}, the HNN extension $M_2$ can be viewed as the corner of an amalgamated free product of tracial von Neumann algebras. Since $P$ has no amenable direct summand, it then follows from \cite[Theorem 4.2]{CH08} that $P' \cap M_2 \prec A_2$. So our claim that $\al(A_1 p_1) \prec A_2$ follows.

By symmetry, we also have the intertwining $\al^{-1}(A_2) \prec A_1 p_1$ inside $p_1 M_1 p_1$. Applying $\al$, we find that $A_2 \prec \al(A_1 p_1)$ inside $M_2$.

Having proven that inside $M_2$ we have the intertwining relations $\al(A_1 p_1) \prec A_2$ and $A_2 \prec \al(A_1 p_1)$, it follows from Theorem \ref{thm.stable} that the equivalence relations $\cR(A_1 p_1 \subset p_1 M_1 p_1)$ and $\cR(A_2 \subset M_2)$ are stably isomorphic. By construction, $\cR(A_1 p_1 \subset p_1 M_1 p_1)$ is the restriction of $\cR(A_1 \subset M_1)$ to the support of $p_1$. So we conclude that the equivalence relations $\cR(A_1 \subset M_1)$ and $\cR(A_2 \subset M_2)$ are stably isomorphic. In particular, these ergodic nonsingular equivalence relations must have the same type. Using Proposition \ref{prop.type}, we find \eqref{desire}.
\end{proof}

\end{document}